\documentclass[11pt,a4paper]{article}

\usepackage[english]{babel}
\usepackage{amsmath,amsfonts,amssymb,latexsym,url}
\usepackage{hyperref}
\usepackage{color}
\usepackage{amsthm}
\usepackage{tikz}
\usepackage{graphicx}

\newtheorem{theorem}{Theorem}[section]
\newtheorem{lemma}[theorem]{Lemma}

\newtheorem{proposition}[theorem]{Proposition}

\newtheorem{definition}[theorem]{Definition}

\DeclareMathOperator{\Aut}{Aut}
\DeclareMathOperator{\Cay}{Cay}
\DeclareMathOperator{\dist}{\partial}
\DeclareMathOperator{\dgr}{dgr}
\DeclareMathOperator{\tr}{tr}

\DeclareMathOperator{\spec}{sp}

\def\e{\mbox{\boldmath $e$}}
\def\j{\mbox{\boldmath $j$}}
\def\m{\mbox{\boldmath $m$}}

\def\w{\mbox{\boldmath $w$}}

\def\vchi{\mbox{\boldmath $\chi$}}

\def\0{\mbox{\boldmath $0$}}

\def\A{\mbox{\boldmath $A$}}
\def\Alg{\mbox{${\cal A}$}}
\def\B{\mbox{\boldmath $B$}}
\def\C{\mbox{\boldmath $C$}}

\def\E{\mbox{\boldmath $E$}}
\def\I{\mbox{\boldmath $I$}}
\def\J{\mbox{\boldmath $J$}}

\def\S{\mbox{\boldmath $S$}}
\def\V{\mbox{\boldmath $V$}}
\def\W{\mbox{\boldmath $W$}}

\def\Par{\pi}
\def\exc{\mbox{$\varepsilon$}}
\def\G{\Gamma}
\def\Re{\mathbb R}
\def\Z{\mathbb Z}

\begin{document}
\title{Quotient-polynomial graphs
}

\author{M.A. Fiol
\\ \\
{\small Universitat Polit\`ecnica de Catalunya} \\
{\small Dept. de Matem\`atica Aplicada IV, Barcelona, Catalonia}\\
{\small E-mail: {\tt fiol@ma4.upc.edu}} \\
 }
\date{}

\maketitle

\begin{abstract}
As a generalization of orbit-polynomial and distance-regular graphs, we introduce
the concept of a quotient-polynomial graph. In these graphs every vertex $u$ induces the same regular partition  around $u$, where all vertices of each cell
are equidistant from $u$. 
Some properties and characterizations of such graphs are studied. For instance, all quotient-polynomial graphs are walk-regular and distance-polynomial. Also,  we show that every quotient-polynomial graph
generates a (symmetric) association scheme.
\end{abstract}



\section{Introduction and preliminaries}
As expected, the most interesting combinatorial structures are those bearing some kind of symmetry and/or regularity. In fact, in general, high symmetry imply high regularity, but the converse does not necessarily holds. Moreover, symmetric structures suggest definitions of new structures obtained, either by relaxing the conditions of symmetry, or having the same regularity properties as the original ones. In turn, the latter can give rise to new definitions by relaxing the mentioned symmetry conditions.
In graph theory, a good example of the above are the distance-transitive graphs, with automorphism group having orbits constituted by all vertices at a given distance. Attending to their symmetry, related concepts are the vertex-symmetric, symmetric, and orbit polynomial graphs \cite{be86,be87}.
Besides, concerning regularity, distance-transitive graphs can be generalized to distance-regular graphs \cite{b93,bcn89}, distance-polynomial graphs, and degree-regular graphs \cite{w82}.
In this paper, we introduce the concept of a quotient-polynomial graph, which could be thought of as the regular counterpart of orbit polynomial graphs.
In a quotient-polynomial graph, every vertex $u$ induces the same regular partition around $u$, with the additional condition that all vertices of each cell
are equidistant from $u$.
Some properties and characterizations of such graphs are studied. For instance, all quotient-polynomial graphs are walk-regular and distance-polynomial. Our study allows us to provide  a characterization of those distance-polynomial and vertex-transitive
graphs which are distance-regular. Also, we show that every quotient-polynomial graph
generates a (symmetric) association scheme.

Throughout this paper, $\G$ denotes a (connected) graph with vertex set $V$, edge set $E$, and diameter $D$. For every $u\in V$ and $i=0,\ldots,D$, let $\Gamma_i(u)$ denote the set of vertices at distance $i$ from $u$, with $\Gamma(u)=\Gamma_1(u)$, and  let $\e_u$ be the  characteristic ($u$-th unitary) vector of $\Gamma_0(u)$. The {\em eccentricity} of $u$, denoted by $\exc(u)$, is the maximum distance  between $u$ and any other vertex $v$ of $\G$. Let $\A_i$ be the $i$-th distance matrix, so that $\A=\A_1$ is the adjacency matrix of $\G$, with spectrum $\spec\G=\{\lambda_0^{m_0}, \ldots, \lambda_d^{m_d}\}$, where $\lambda_0>\lambda_1>\cdots >\lambda_d$, and the superscripts $m_i=m(\lambda_i)$ stand for the multiplicities. Let $\E_j$, $j=0,\ldots,d$ be the minimal idempotents representing the orthogonal projections on the $\lambda_j$-eigenspaces . Let $\Alg(\G)=\Re_d[\A]$ be the {\em adjacency algebra} of $\G$, that is, the algebra of all polynomials in $\A$ with real coefficients.

Following Fiol, Garriga and Yebra \cite{fg97,fgy96}, the $uv$-entry of  $\E_j$ is referred to as the {\it crossed $(uv$-$)$local
multiplicity\/} of the eigenvalue $\lambda_j$,  and it is denoted by
$m_{uv}(\lambda_j)$. In particular, for a regular graph on $n$ vertices, $\E_0=\frac{1}{n}\J$ and, hence, $m_{uv}(\lambda_0)=1/n$ for every $u,v\in V$.
Since $\A^{\ell}=\sum_{j=0}^d \lambda_j^{\ell} \E_j$, the number of walks of
length $\ell$ between two vertices $u,v$ is
\begin{equation}
\label{crossed-mul->num-walks}
a_{uv}^{({\ell})} =(\A^{\ell})_{uv}=
\sum_{j=0}^d m_{uv}(\lambda_j)\lambda_j^{\ell} \qquad (\ell \geq 0).
\end{equation}
In particular, the {\em $(u$-$)$local multiplicities} are $m_{u}(\lambda_i)=\|\E_i\e_u\|^2=(\E_i)_{uu}$, $i=0,\ldots,d$, and satisfy $\sum_{i=0}^d m_u(\lambda_i)=1$ and $\sum_{u\in V}m_u(\lambda_i)=m_i$, $i=0,\ldots,d$.

A graph $\G$ with diameter $D$ is called {\em $h$-punctually walk-regular}, for some $h=0,\ldots,D$, when the number of walks $a_{uv}^{({\ell})}$ for any pair of vertices $u,v$ at distance $h$ only depends on $\ell$.
From the above, this means that the crossed local multiplicities $m_{uv}(\lambda_j)$ only depend on $\lambda_j$ and we write it as $m_{h}(\lambda_j)$ (see
Dalf\'{o}, Van Dam, Fiol,  Garriga, and
Gorissen \cite{ddfgg11} for more details). Notice that, in particular, a $0$-punctually walk-regular graph is the same as a walk-regular graph, a concept introduced by Godsil and
McKay  \cite{gmc80} .

A partition $\Par=\{V_1,\ldots, V_m\}$ of the
vertex set $V$ is called {\em regular} (or {\em equitable})
whenever for any $i,j=1,\ldots,m$, the {\em intersection numbers} $b_{ij}(u)=\G(u)\cap V_j$, where $u\in V_i$, do not depend on the vertex $u$ but only on the subsets ({\em classes} or {\em cells}) $V_i$ and $V_j$. In this case, such numbers are simply written as $b_{ij}$, and the $m\times m$ matrix $\B=(b_{ij})$ is referred to as the {\em quotient matrix} of $\A$ with respect to $\Par$.

The {\em characteristic matrix} of (any) partition $\Par$ is the $n\times m$ matrix $\S=(s_{ui})$ whose $i$-th column is the characteristic vector of $V_i$, that is, $s_{ui}=1$ if $u\in V_i$, and $s_{ui}=0$ otherwise. In terms of such a matrix, it is known that $\Par$ is regular if and only if there exists an $m\times m$ matrix $\C$ such that
$$
\S\C=\A\S.
$$
Moreover, in this case, $\C=\B$, the quotient matrix of $\A$ with respect to $\Par$. Then, using this it easily follows that all the eigenvalues of $\B$ are also eigenvalues of $\A$.
For more details, see  Godsil \cite{g93}.

\section{Partitions around a vertex}

In this section we introduce several types of partitions bearing some regularity properties with respect to a given vertex $u$. We begin by considering those partitions where all vertices of the  same class are equidistant from $u$.

\begin{definition}
Let $\G$ have diameter $D$.
Given a vertex $u$, a  {\em ($u$-)distance-faithful} partition around $u$, denoted by $\Par(u)$, is a partition $V_0,V_1,\ldots,V_r$, with $r\ge \exc(u)$ such that $V_0=\{u\}$ and, for $i=1,\ldots,r$, every pair of vertices  $v,w\in V_i$ are at the same distance from $u$: $\dist(u,v)=\dist(u,w)$.
\end{definition}
Thus, in particular, $\Par(u)$ is a {\em distance partition around} $u$ whenever $v,w\in V_i$ if and only if $\dist(u,v)=\dist(u,w)=i$. In other words, $V_i=\Gamma_i(u)$ for every $i=0,\ldots,r$ and, hence, $r=\exc(u)$.


For every pair of vertices $u,v\in V$ we consider the vectors
of crossed local multiplicities
$$
\m(u,v)=((\E_0)_{uv},(\E_1)_{uv},\ldots,(\E_d)_{uv})
=(m_{uv}(\lambda_0),m_{uv}(\lambda_1),\ldots,m_{uv}(\lambda_d)),
$$
and numbers of $\ell$-walks for $\ell=0,\ldots,d$ between $u$ and $v$
$$
\w(u,v)=((\A^0)_{uv},(\A^1)_{uv},\ldots,(\A^{d})_{uv})
=(a^{(0)}_{uv},a^{(1)}_{uv},\ldots,a^{(d)}_{uv}).
$$
The following result is an immediate consequence of \eqref{crossed-mul->num-walks} (see, for instance, \cite{ddfgg11}).
\begin{lemma}
Given some vertices $u,v,x,y$, we have $\m(u,v)=\m(x,y)$ if and only if
$\w(u,v)=\w(x,y)$.
\end{lemma}
Then, we can define the two following equivalent concepts:
\begin{definition}
A partition $\Par(u)=\{U_0,U_1,\ldots, U_r\}$ is {\em $(u$-$)$walk-regular} $($or {\em $(u$-$)$-spectrum-regular}$)$ around a vertex $u\in V$ if, for every $i=0,\ldots,r$ the set $U_i$ is constituted by all vertices $v\in V$ with the same vector  $\w(u,v)$  $($or, equivalently, $\m(u,v))$.
\end{definition}
Now we will prove that, if a partition is both regular and $u$-distance-faithful, then it is  also $u$-walk-regular.
Before that, we have the following straightforward lemma.

\begin{lemma}
\label{wr->df}
Every walk-regular partition $\Par(u)$  around a vertex $u\in V$ is also distance-faithful around the same vertex.
\end{lemma}
\begin{proof}
By contradiction, assume that $v,w\in U_i$ and $\dist(u,v)>\dist(u,w)=\ell$. Then, we would have $a_{uw}^{(\ell)}\neq 0$ but $a_{uv}^{(\ell)}= 0$, against the hypothesis of $u$-walk-regularity.
\end{proof}

\begin{proposition}
Let $\sigma(u)=\{V_0,\ldots,V_s\}$ be a $u$-distance-faithful and regular partition. Then $\sigma(u)$ defines a  $u$-walk-regular partition $\Par(u)=\{U_0,\ldots,U_r\}$ with $r\le s$ (by the union of some sets $V_i$, if necessary).
\end{proposition}
\begin{proof}
To prove that the number of walks $a_{uv}^{(\ell)}$, with $v\in V_i$, only depend on $i$ and $\ell$, we use induction on $\ell$. The result is clear for $\ell\le 1$. Now suppose that the result holds for some $\ell>1$. Then, for a given $v\in V_i$,
\begin{equation}
\label{lin-sys}
a_{uv}^{(\ell+1)}=\sum_{j=0}^r b_{ji} a_{uv}^{(\ell)},\qquad i=0,\ldots,r.
\end{equation}
and, hence, $a_{uv}^{(\ell+1)}=a_i^{(\ell)}$ does not depend on $v$. Finally, if there are sets $V_{i_1},V_{i_2},\ldots$ with vertices $v$ having the same vector $\w(u,v)$, we consider their union $U_i=V_{i_1}\cup V_{i_2}\cup\cdots$ to form the claimed $u$-walk-regular partition.
\end{proof}

To prove the converse, we need an extra hypothesis, which in fact leads to a stronger result in terms of the new concept defined below.

Given a vertex $u$ of a graph $\Gamma$, the so-called {\em $u$-local spectrum} is constituted by those eigenvalues $\lambda_i$ of $\Gamma$ such that $\E_i\e_u\neq \0$ (that is, with nonzero $u$-local multiplicity $m_u(\lambda_i)$). Moreover, these are referred to as the {\em $u$-local eigenvalues}.
Let us consider the vector space  $\Alg(u)$ spanned by the vectors $\A^{\ell}\e_u$, $\ell=0,\ldots,d$. Then, it is known that $\Alg(u)$ has dimension $d_u+1$ and basis $\e_u, \A\e_u,\ldots, \A^{d_u}\e_u$; see e.g. \cite{fgy96,fg97}.

\begin{definition}
\label{defi1}
Let $u$ be a vertex with $d_u+1$ distinct local eigenvalues. Let $\Par(u)=\{U_0,\ldots,U_r\}$ be a $u$-walk-regular partition, with $\vchi_i$ being the characteristic vector of $U_i$ $($note that $\vchi_0=\e_u$$)$. Then $\Par(u)$ is said to be {\em quotient-polynomial} whenever $\vchi_i\in \Alg(u)$ for every $i=0,\ldots,d_u$.
\end{definition}

In the following result the walk-regular partitions that are quotient-polynomial (and regular) are characterized.

\begin{theorem}
\label{ineq&equ}
Let $u$ be a vertex with  $d_u+1$ distinct local eigenvalues. Let $\Par(u)=\{U_0,\ldots,U_r\}$ be a $u$-walk-regular partition. Then,
\begin{equation}
\label{main-ineq}
r\ge d_u,
\end{equation}
with equality if and only if $\Par(u)$ is a quotient-polynomial partition.
Moreover, in this case  $\Par(u)$ is also regular.
\end{theorem}

\begin{proof}
For $i=0,\ldots,r$, let $a_i^{(\ell)}$ the common value of the number of $\ell$-walks from $u$ to every $v\in U_i$, and
let us consider the vector space $\Alg(\Par(u))=\langle \vchi_0(=\e_u), \vchi_1,\ldots,\vchi_{r}\rangle$. Then, as
\begin{equation}
\label{powers-vs-intersec}
 \A^{\ell}\e_u=a_0^{(\ell)}\vchi_0+\cdots+a_{r}^{(\ell)}\vchi_{r}\qquad (\ell\ge 0),
\end{equation}
we have that $\Alg(u)\subset\Alg(\Par(u))$ and, hence,
\begin{equation}
\label{ineq0}
d_u+1=\dim\Alg(u)\le  \dim  \Alg(\Par(u))=r+1,
\end{equation}
which proves \eqref{main-ineq}. Of course, the same conclusion can be reached by considering  the common value $m_{ij}$ of the crossed local multiplicities $m_{uv}(\lambda_j)$, for every $v\in U_i$. Then,
$$
\E_j\e_u=m_{0j}\vchi_0+\cdots+m_{dj}\vchi_d\qquad (0\le j\le d).
$$
(Only $d_u+1$ of the above equations are not trivially null).

If $r=d_u$, we have that $\Alg(u)=\Alg(\Par(u))$, that is, every vector $\vchi_i$ is a linear combination of the vectors $\A^{\ell}\e_u$ for $i,\ell=0,\ldots,r$, and $\Par(u)$ is a quotient-polynomial partition.
Conversely, if  $\Par(u)$ is quotient-polynomial,  we have $\Alg(\Par(u))\subseteq \Alg(u)$. Hence, $r\le d_u$ which, together with \eqref{ineq0} ($\Par(u)$ is also walk-regular), leads to $r=d_u$.

In fact, the constants of the above linear combinations, which are the coefficients of the polynomials
$p_i(x)=\omega_{i0}+\omega_{i1}x+\cdots+\omega_{ir}x^r$, can be computed in the following way:
The $r+1$ first equations in \eqref{powers-vs-intersec} are, in matrix form,
$$
\left(
\begin{array}{cccc}
a_0^{(0)} & 0 & \cdots & 0 \\
a_0^{(1)} & a_1^{(1)} & \cdots & a_r^{(1)}\\
 & & \ddots & \\
a_0^{(r)} & a_1^{(r)} & \cdots & a_r^{(r)}\\
\end{array} \right)\left(
\begin{array}{c}
\e_u\\
\vchi_1\\
\vdots \\
\vchi_r
\end{array} \right)
=\left(
\begin{array}{c}
\e_u\\
\A\e_u\\
\vdots \\
\A^r\e_u
\end{array} \right).
$$
But the coefficient matrix $\W$ with entries  $(\W)_{\ell i}=a_{i}^{(\ell)}$, for $\ell,i=0,1,\ldots,r$, is a change-of-basis matrix and, hence, it is invertible.
As a consequence, for every $i=0,1,\ldots, r$, the coefficients $\omega_{i0}$, $\omega_{i0}$,\ldots, $\omega_{i0}$ of $p_i$ correspond to the $i$-th row of  $\W^{-1}$.

Finally,  to prove that $\Par(u)$ is regular, let us choose one vertex $u_i$ in each $U_i$, $i=0,1,\ldots, r$ and  consider the  $(r+1)\times (r+1)$ matrices $\B$ and  $\W^+$,  with entries  $(\B)_{ij}=|\Gamma_1(v_i)\cap U_j|$ and $(\W^+)_{\ell i}=a_{i}^{(\ell+1)}$, $\ell,i,j=0,1,\ldots,r$, respectively. Then, the $r+1$ equations of \eqref{lin-sys} can be written as
$$
\W\B^{\top}=\W^+.
$$
Hence, the entries $b_{ji}$ of the matrix $\B^{\top}=\W^{-1}\W^+$ do not depend on the chosen vertices $v_i$, and the partition is regular with quotient matrix $\B=(b_{ij})$.
\end{proof}

All the above result can be summarized in the followin theorem.
\begin{theorem}
Let $u$ be a vertex of a  graph $\G$.
Let $\Par(u)$ be a partition  with $r+1$ classes around $u$ having $d_u+1$ distinct local eigenvalues $\mu_0>\mu_1>\cdots>\mu_{d_u}$.
Then, the following assertions are equivalent:
\begin{itemize}
\item[$(a)$]
The partition $\Par(u)$ is quotient-polynomial.
\item[$(b)$]
The partition $\Par(u)$ is $u$-walk-regular with $r=d_u$.
\item[$(c)$]
There exist  polynomials $p_i$ with $\deg p_i\le r$, $i=0,\ldots,r$, such that $\vchi_i=p_i(\A)\e_u$.
 \item[$(d)$]
$\Alg(u)=\Alg(\Par(u))$ with basis $\vchi_0(=\e_u),\vchi_1,\ldots,\vchi_r$.
\end{itemize}
\end{theorem}
\begin{proof}
By Theorem \ref{ineq&equ}, we only need to prove
$(c)\Rightarrow (d)$: From $(c)$ we have that $\Alg(\Par(u))\subset \Alg(u)$ and hence $r\le d_u$. Also, as $\sum_{i=0}^r \vchi_i=q(\A)\e_u=\j$, where $q=\sum_{i=0}^r p_i$,
we get $q(\A)(\A-k\I)\e_u=\0$, where the left-hand side is a polynomial with degree at most $r+1$.
Moreover,  $m(x)=\prod_{i=0}^{d_u}(x-\mu_i)$, with $\dgr m=d_u+1$, is the polynomial of minimum degree satisfying $m(\A)\e_u=\0$
(see \cite{fgy96}). Therefore, $r\ge d_u$ and, hence, $r=d_u$, $\Alg(u)=\Alg(\Par(u))$, and  $\e_0,\vchi_1,\ldots,\vchi_r$ is a basis.
\end{proof}

\section{Quotient-polynomial graphs}
\label{Qp-graphs}
In this section we study the graphs having the same (i.e. with the same parameters) quotient-polynomial partition around each of their vertices. With this aim, we now follow a global approach.

Let $\G$ be a graph with vertex set $V$, $d+1$ distinct eigenvalues, and adjacency algebra $\Alg(\G)=\langle \I,\A,\ldots,\A^d\rangle$.
Then a partition $\Par(J)=\{J_0,J_1,\ldots, J_r\}$ of $V\times V$  is called {\em walk-regular} whenever each $J_i$ is the set with elements $(u,v)$ having identical vector $\m(u,v)$ (or $\w(u,v)$).
So, from Lemma \ref{wr->df}, all pairs of vertices in a given $J_i$ are at the same distance, and we assume that the pairs in $J_0$ are of the form $(u,u)$ (distance zero).
Let $\J_{i}$, $i=0,\ldots,r$, the $n\times n$ matrices, indexed by the vertices of $\G$,
representing the equivalence classes $J_i$, that is,
\begin{equation}
\label{equiv-matrices}
(\J_i)_{uv}=\left\{
\begin{array}{ll}
1 & \mbox{{\rm if $(u,v)\in J_i$,}}\\
0 & \mbox{\rm otherwise.}
\end{array}\right.
\end{equation}
Let $J_h$ be an equivalence class with elements $(u,v)$ satisfying $\dist(u,v)=h$. Then,  $\J_h=\A_h$ if and only if $\G$ is $h$-punctually walk-regular and, in particular, $\J_0=\I$ if and only if $\G$ is walk-regular.

From these matrices we can now define our main concept:
\begin{definition}
A graph $\G$, with walk-regular partition $\Par(J)=\{J_0,J_1,\ldots, J_r\}$ and adjacency algebra $\Alg(\G)$, is {\em quotient-polynomial} if $\J_i\in \Alg(\G)$ for every $i=0,\ldots,r$.
\end{definition}

Thus, $\G$ is quotient-polynomial if and only if there exist polynomials $p_i$, with $\deg p_i\le d$, such that $p_i(\A)=\J_i$, $i=0,\ldots,r$ (this inspired our definition).
In fact, the following result shows that this only happens when $r=d$.
We omit its proof since it goes along the same lines of reasoning as that of theorem \ref{ineq&equ}.

\begin{theorem}
\label{ineq&equ}
Let $\G$ be a graph as above. Let $\Par(J)=\{J_0,\ldots,J_r\}$ be a walk-regular partition. Then,
\begin{equation}
\label{ineq}
r\ge d,
\end{equation}
and equality occurs if and only if $\G$ is quotient-polynomial.
\end{theorem}

Since $\J_0+\cdots+\J_d=\J$, the all-1 matrix, the sum polynomial $H=\sum_{i=0}^d p_i$ is the Hoffman polynomial satisfying $H(\A)=\J$. Therefore, a quotient-polynomial graph is connected and regular (see Hoffman \cite{hof63}). Moreover, the same reasoning used in \cite{be86}[Th. 2.4] to prove that every orbit polynomial graph is vertex transitive, shows that every quotient-polynomial graph is walk-regular, that is, $\J_0=\I$. Indeed, if $\J_0\neq \I$,  the equality $\J_0\A=\A\J_0$ ($\Alg(\G)$ is a commutative algebra) leads to a contradicion because $\G$ is connected.

Thus, for any vertex $u$, the {\em induced partition} $\Par(u)$ of $V$, with characteristic vectors $\e_u$, $\J_1\e_u$,\ldots, $\J_d\e_u$ is quotient-polynomial since
$$
\J_i=p_i(\A)\qquad \Rightarrow \qquad \vchi_i=\J_i\e_u=p_i(\A)\e_u,\qquad i=0,\ldots,d.
$$

Then, we can summarize all the above results in the `global analogue'  of Theorem \ref{ineq&equ}.

\begin{theorem}
\label{main-theo}
Let $\G$ be a graph with vertex set $V$, and $d+1$ distinct eigenvalues.
Let $\Par(V)=\{J_0,\ldots,J_r\}$ be a partition of $V\times V$ with $J_0=\{(u,u):u\in V\}$.
Then, the following assertions are equivalent:
\begin{itemize}
\item[$(a)$]
$\G$ is a quotient-polynomial graph.
\item[$(b)$]
The partition $\Par(V)$ is walk-regular with $r=d$.
\item[$(c)$]
There exist  polynomials $p_i$ with $\deg p_i\le r$, $i=0,\ldots,r$,
such that $\J_i=p_i(\A)$.
\item[$(d)$]
For every vertex $u$, the induced partition $\Par(u)$ of $V$ is quotient-polynomial with the same polynomials $p_i$.
\item[$(e)$]
$\Alg(\G)=\Alg(\Par(V))$ with basis $\J_0(=\I),\J_1,\ldots,\J_r$.\hfill $\square$
\end{itemize}
\end{theorem}
(Notice that the first equality in $(e)$ implies $r=d$.)



\subsection{An Example and some more details}
Let us consider the following example of quotient-polynomial graph:
The circulant graph $\G=\Cay(\Z_{17};\pm 1,\pm 4)$ has vertices $V=\Z_{17}$ and vertex $u$ is adjacent to vertices  $u\pm 1$ and $u\pm 4$. Then, $\G$ is a 4-regular vertex-transitive graph with diameter $D=3$, and spectrum (with numbers rounded to three decimals)
$$
\spec\G=\{4,2.049^4,0.344^4,-2.906^4,-0.488^4\}
$$

\begin{figure}[t]
\begin{center}
\includegraphics[scale=0.8]{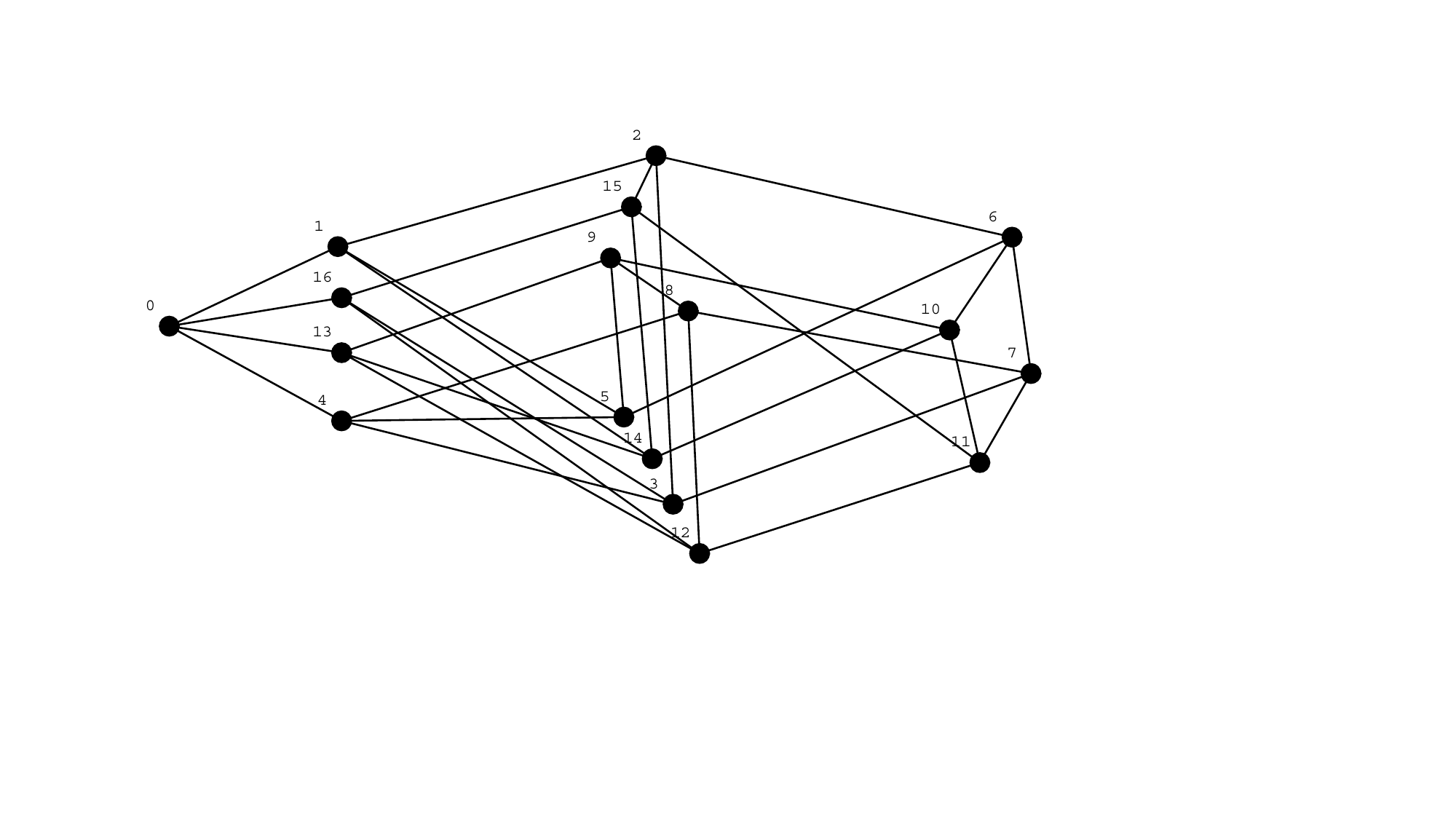}
\vskip -2.5cm
\caption{Circulant graph $\Cay(\Z_{17};1,4)$.}
\label{fig1}
\end{center}
\end{figure}

As shown in Fig. \ref{fig1}, the walk-regular partition around vertex $0$ has clases
$U_0=\{0\}$, $U_1=\{1,4,13,16\}$, $U_2=\{3,5,12,14\}$, $U_3=\{2,8,9,15\}$, and $U_4=\{6,7,10,11\}$.
(The corresponding intersection diagram is shown in Fig. 2.) Then the matrices $\W$ and $\W^+$ of numbers of walks from $0$ to a vertex of $U_i$ are
$$
\W=\left(
\begin{array}{ccccc}
1 & 0 & 0 & 0 & 0 \\
0 & 1 & 0 & 0 & 0 \\
4 & 0 & 2 & 1 & 0 \\
0 & 9 & 1 & 3 & 3 \\
36 & 5 & 24 & 16 & 10
\end{array}
\right), \qquad
\W^+=\left(
\begin{array}{ccccc}
0 & 1 & 0 & 0 & 0 \\
4 & 0 & 2 & 1 & 0 \\
0 & 9 & 1 & 3 & 3 \\
36 & 5 & 24 & 16 & 10\\
20 & 100 & 36 & 55 & 60 \\
\end{array}
\right),
$$
Then, from the inverse of $\W$ we obtain
the  quotient polynomials:
\begin{align*}
p_0(x) & = 1,\\
p_1(x) & = x,\\
p_2(x) & = \frac{1}{26}(3x^4-10x^3-10x^2+75x-36),\\
p_3(x) & = \frac{1}{13}(-3x^4+10x^3+31x^2-75x-16),\\
p_4(x) & = \frac{1}{26}(5x^4-8x^3-56x^2+47x+44).
\end{align*}
whereas the transpose of the intersection matrix turns out to be
$$
\B^{\top}=\W^{-1}\W^+=
\left(
\begin{array}{ccccc}
0 & 1 & 0 & 0 & 0 \\
4 & 0 & 2 & 1 & 0 \\
0 & 2 & 0 & 1 & 1 \\
0 & 1 & 1 & 1 & 1 \\
0 & 0 & 1 & 1 & 2
\end{array}
\right)
$$

As in the case of distance-regular graph, the entries $b_{ij}=(\B)_{ij}$ (shown also in Fig. 2) can also be calculated by using the formulas
$$
b_{ij}=p_{1i}^j=\frac{\tr(\V_1\V_i\V_j)}{\tr(\V_i^2)}=\frac{\langle p_1p_i,p_j\rangle_{\G}}{\|p_j\|_{\G}^2},
$$
where we use the scalar product
$$
\langle f,g\rangle_{\G}=\frac{1}{n}\tr(f(\A)g(\A))=\frac{1}{n}\sum_{i=0}^d m_i f(\lambda_i)g(\lambda_i).
$$
Note that, since $\tr(\J_i\J_j)=0$ for $i\neq j$,  the quotient polynomials $p_i$ are orthogonal with respect to such a product.


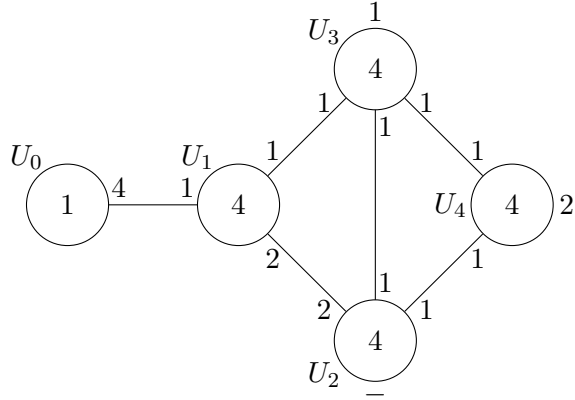
\begin{figure}[h]
\begin{center}
\begin{tikzpicture}[scale=.9]
\draw (1,3) circle [radius=0.6];
\node  at (0.4,3.7) {$U_0$};
\node  at (1.75,3.25) {$4$};
\draw (3.5,3) circle [radius=0.6];
\node  at (2.9,3.7) {$U_1$};
\node  at (2.75,3.25) {$1$};
\node  at (4,3.8) {$1$};
\node  at (4,2.2) {$2$};
\draw (5.5,1) circle [radius=0.6];
\node  at (4.75,0.5) {$U_2$};
\node  at (4.75,1.45) {$2$};
\node  at (5.65,1.85) {$1$};
\node  at (6.25,1.45) {$1$};
\node  at (5.5,0.20) {$-$};
\draw (5.5,5) circle [radius=0.6];
\node  at (4.75,5.55) {$U_3$};
\node  at (4.75,4.5) {$1$};
\node  at (5.65,4.15) {$1$};
\node  at (6.25,4.5) {$1$};
\node  at (5.5,5.85) {$1$};
\draw (7.5,3) circle [radius=0.6];
\node  at (6.6,3) {$U_4$};
\node  at (7,3.8) {$1$};
\node  at (7,2.2) {$1$};
\node  at (8.3,3) {$2$};
\draw (1.6,3)--(2.9,3);
\draw (3.93,2.57)--(5.07,1.43);
\draw (3.93,3.43)--(5.07,4.57);
\draw (5.5,1.6)--(5.5,4.4);
\draw (5.93,1.43)--(7.07,2.57);
\draw (5.93,4.57)--(7.07,3.43);
\node  at (1,3) {$1$};
\node  at (3.5,3) {$4$};
\node  at (5.5,1) {$4$};
\node  at (5.5,5) {$4$};
\node  at (7.5,3) {$4$};
\end{tikzpicture}
\label{fig2}
\end{center}
\caption{Intersection diagram of $\Cay(\Z_{17};1,4)$.}
\end{figure}

%
%
%
%
%
%
%

\section{Related concepts}
In this section we study the relationships of quotient-polynomial graphs with other known combinatorial structures.

\subsection{Distance-regular graphs}
Delsarte \cite{d73} proved that a graph $\G$ with $d+1$ distinct eigenvalues is distance-regular if and only if, for every $i=0,\ldots,d$, $\A_i=p_i(\A)$ for some polynomial $p_i$ of degree $i$; see also Weischel \cite{w82}.
Then, as $D\le d\le r$, the following result is clear.

\begin{proposition}
A quotient-polynomial graph $\G$, with diameter $D$ and $r$ classes, is distance-regular if and only if $D=r$.
\hfill $\square$
\end{proposition}

Two generalizations of distance-regular graphs are now considered in the two following subsections.

\subsection{Distance-polynomial graphs}
Inspired by the above characterization of distance-regular graphs, Weichsel \cite{w82} defined a graph $\G$ with diameter $D$ to be {\em distance-polynomial} if $\A_0,\ldots,\A_D\in \Alg (\G)$. (That is $\A_i=p_i(\A)$  for every $i=0,\ldots,D$ with no condition on the degrees of the polynomials $p_i$.)
Then any distance-regular graph is also distance-polynomial, but the converse does not hold. For instance, any regular graph with diameter two is easily proved to be distance-polynomial (but not necessarily strongly regular). The next result gives a condition for having the equivalence.

\begin{proposition}
A distance-polynomial graph $\G$, with diameter $D$ and $d+1$ distinct eigenvalues, is distance-regular if and only if $D=d$.
\end{proposition}
\begin{proof}
Since every distance-regular graph is distance-polynomial and has diameter $D=d$,
necessity is clear.
To prove sufficiency, assume that $\G$ is distance-polynomial with $D=d$. Then $\A_D=\A_d\in \Alg(\G)$ which, as it was proved in \cite{fgy96,ddfgg11}, implies that $\G$ is distance-regular.
\end{proof}

In our context,

\begin{proposition}
Let $\G$ be a quotient-polynomial graph with diameter $D$ and $d+1$ distinct eigenvalues. Then,
$\G$ is also distance-polynomial.
\end{proposition}
\begin{proof}
Let $\G$ have the  walk-regular partition $\Par(V)=\{J_0,\ldots,J_d\}$ with corresponding matrices $\J_j=q_j(\A)$ for some polynomials $q_j$, $j=0,\ldots,d$.  Then, the polynomials
$$
p_i=\sum_{j\, :\, tr(\mbox{\scriptsize $\A_i\J_j$})\neq 0} q_j,\qquad i=0,\ldots,D,
$$
with $D\le d$, clearly satisfy $p_i(\A)=\A_i$ for $i=0,\ldots,D$, and the result follows.
\end{proof}

\begin{figure}[t]
\begin{center}
\includegraphics[scale=.9]{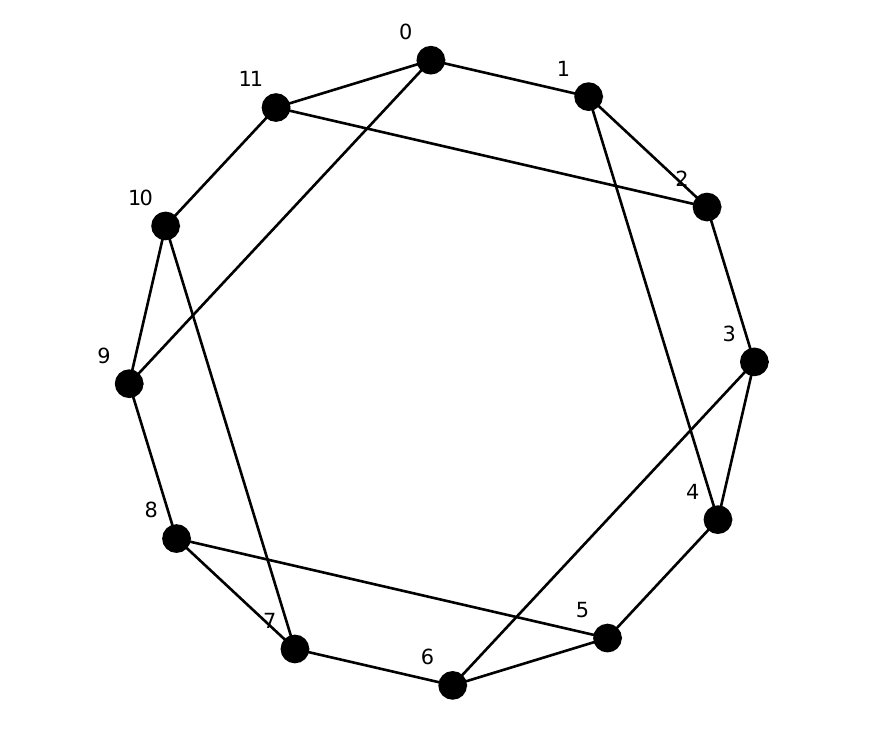}
\caption{Chordal ring $(12,4)$ or prism $Y_6\cong K_2\square C_6$.}
\label{fig3}
\end{center}
\end{figure}

However, the converse result does not hold, even if we require $\G$ to have the same equitable walk-regular partition around each of its vertices.
A counterexample is the so-called `chordal ring' $(12,4)$ or prism $Y_6$ shown in Fig. \ref{fig3}.
This (bipartite) graph has
$$
\spec Y_6=\{3,2^2,1^1,0^4,-1^1,-2^2,-3\},
$$
and it has the walk-regular (and regular) partition with intersection
diagram shown in Fig. \ref{fig4}.

\begin{figure}[h]
\begin{center}
\begin{tikzpicture}[scale=.9]

\draw (1,3) circle [radius=0.6];
\node  at (1,3) {$1$};
\node  at (1.5,3.75) {$2$};
\node  at (1.5,2.25) {$1$};

\draw (3,1) circle [radius=0.6];
\node  at (3,1) {$1$};
\node  at (2.25,1.5) {$1$};
\node  at (3.75,0.75) {$2$};

\draw (3,5) circle [radius=0.6];
\node  at (3,5) {$2$};
\node  at (2.25,4.5) {$1$};
\node  at (3.5,4.25) {$1$};
\node  at (3.75,5.25) {$1$};

\draw (7,1) circle [radius=0.6];
\node  at (7,1) {$2$};
\node  at (6.25,0.75) {$1$};
\node  at (6.5,1.75) {$1$};
\node  at (7.75,0.75) {$1$};

\draw (7,5) circle [radius=0.6];
\node  at (7,5) {$2$};
\node  at (6.25,5.25) {$1$};
\node  at (7.75,5.25) {$1$};
\node  at (7.5,4.25) {$1$};

\draw (11,1) circle [radius=0.6];
\node  at (11,1) {$2$};
\node  at (10.25,0.75) {$1$};
\node  at (10.5,1.75) {$1$};
\node  at (11.75,1.5) {$1$};

\draw (11,5) circle [radius=0.6];
\node  at (11,5) {$1$};
\node  at (10.25,5.25) {$2$};
\node  at (11.75,4.5) {$1$};

\draw (13,3) circle [radius=0.6];
\node  at (13,3) {$1$};
\node  at (12.5,3.75) {$1$};
\node  at (12.5,2.25) {$2$};
\draw (1.43,2.57)--(2.57,1.43);
\draw (1.43,3.43)--(2.57,4.57);

\draw (3.6,1)--(6.4,1);
\draw (3.6,5)--(6.4,5);
\draw (3.43,4.57)--(6.57,1.43);

\draw (7.6,1)--(10.4,1);
\draw (7.43,4.57)--(10.57,1.43);
\draw (7.6,5)--(10.4,5);

\draw (11.43,1.43)--(12.57,2.57);
\draw (11.43,4.57)--(12.57,3.43);

\end{tikzpicture}
\caption{Intersection diagram of $Y_6$.}
\label{fig4}
\end{center}
\end{figure}
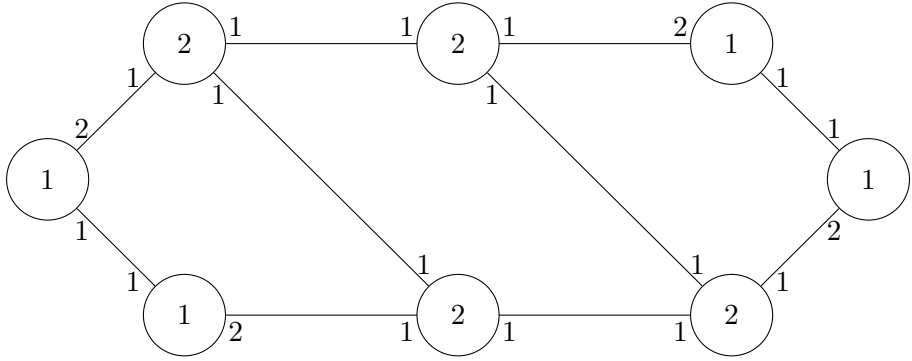

Moreover it is distance-polynomial with polynomials
\begin{align*}
p_0(x) & = 1,\\
p_1(x) & = x,\\
p_2(x) & = \frac{1}{30}(2x^6-25x^4+83x^2-60),\\
p_3(x) & = \frac{1}{20}(x^5-5x^3-16x),\\
p_4(x) & = \frac{1}{20}(-x^6+15x^4-54x^2+20).
\end{align*}
For instance, the matrices $\A_1=\A$ and $\A_4=p_4(\A)$ turn out to be
$$
\left(
\begin{array}{cccccccccccc}
0& 1& 0& 0& 0& 0& 0& 0& 0& 1& 0& 1\\
1& 0& 1& 0& 1& 0& 0& 0& 0& 0& 0& 0\\
0& 1& 0& 1& 0& 0& 0& 0& 0& 0& 0& 1\\
0& 0& 1& 0& 1& 0& 1& 0& 0& 0& 0& 0\\
0& 1& 0& 1& 0& 1& 0& 0& 0& 0& 0& 0\\
0& 0& 0& 0& 1& 0& 1& 0& 1& 0& 0& 0\\
0& 0& 0& 1& 0& 1& 0& 1& 0& 0& 0& 0\\
0& 0& 0& 0& 0& 0& 1& 0& 1& 0& 1& 0\\
0& 0& 0& 0& 0& 1& 0& 1& 0& 1& 0& 0\\
1& 0& 0& 0& 0& 0& 0& 0& 1& 0& 1& 0\\
0& 0& 0& 0& 0& 0& 0& 1& 0& 1& 0& 1\\
1& 0& 1& 0& 0& 0& 0& 0& 0& 0& 1& 0
\end{array}
\right),\quad
\left(
\begin{array}{cccccccccccc}
0& 0& 0& 0& 0& 0& 1& 0& 0& 0& 0& 0\\
0& 0& 0& 0& 0& 0& 0& 1& 0& 0& 0& 0\\
0& 0& 0& 0& 0& 0& 0& 0& 1& 0& 0& 0\\
0& 0& 0& 0& 0& 0& 0& 0& 0& 1& 0& 0\\
0& 0& 0& 0& 0& 0& 0& 0& 0& 0& 1& 0\\
0& 0& 0& 0& 0& 0& 0& 0& 0& 0& 0& 1\\
1& 0& 0& 0& 0& 0& 0& 0& 0& 0& 0& 0\\
0& 1& 0& 0& 0& 0& 0& 0& 0& 0& 0& 0\\
0& 0& 1& 0& 0& 0& 0& 0& 0& 0& 0& 0\\
0& 0& 0& 1& 0& 0& 0& 0& 0& 0& 0& 0\\
0& 0& 0& 0& 1& 0& 0& 0& 0& 0& 0& 0\\
0& 0& 0& 0& 0& 1& 0& 0& 0& 0& 0& 0
\end{array}
\right),
$$
respectively. (Notice that $Y_6$ is antipodal.)
Hovewer, as $r>d$, this graph is not quotient-polynomial.

\subsection{Orbit polynomial graphs}
The definition of orbit polynomial graph $\G$, due to Beezer \cite{be86,be87} is similar to that of quotient polynomial graph, but now the classes $J_0,\ldots,J_r$ of the partition of $V\times V$ are the orbits of the action of the automorphism group $\Aut\G$ on such a set. Namely
$(u,v)$ and $(w,x)$ are in the same equivalence class $J_i$ if and only if there is some $\sigma\in \Aut\G$
such that $(w,x)=(\sigma(u),\sigma(v))$. Then, the $n\times n$ matrices, $\J_0,\ldots,\J_r$,
representing the equivalence classes $J_i$, are defined as in \eqref{equiv-matrices}, and $\G$ is said to be orbit polynomial whenever $\J_i\in\Alg(\G)$ for all $i=0,\ldots,r$.

Since every automorphism preserves distance and number of walks between vertices, the following result is clear:

\begin{proposition}
Every orbit polynomial graph is a quotient-polynomial graph.\hfill{ }$\square$
\end{proposition}

%

\subsection{Association schemes}
Recall that a (symmetric) association scheme $\Alg$ with $d$ classes can be defined
as a set of $d$ graphs $\G_i = (V,E_i)$, $i=1,\ldots, d$, on the same vertex set $V$, with adjacency matrices $\A_i$ satisfying
$\sum_{i=0}^d \A_i = \J$, with $\A_0 := I$,  and $\A_i\A_j =\sum_{k=0} p^k_{ij} \A_k$, for
some integers $p^k_{ij}$, $i,j,k=0,\ldots,d$. Then, following Godsil [24] , we say that the graph $\G_i$ is the $i$-th class of the scheme, and so we indistinctly use the words ``graph" or ``class" to mean the same thing.
(For more details, see \cite{bcn89,d73,g93}.)

A $d$-class association scheme $\Alg$ is said to be {\em generated} by one of its matrices $\A_i$ (or graph $\G_i$) if it determines the other relations, that is,
the powers $\I,\A_i,\ldots,\A_i^d$ and $\J$ span the Bose-Mesner algebra of $\Alg$.

In particular, if $\G_i$ is connected, then it generates the whole scheme if and only if it has $d+1$ distinct eigenvalues. Then, since a quotient polynomial graph $\G$  is connected, Theorem \ref{main-theo}$(e)$ yields:

\begin{theorem}
Let $\G$ be a graph with $d+1$ distinct eigenvalues. Then, $\G$ is the connected generating graph of a $d$-class association scheme $\Alg$ if and only if $\G$ is a quotient-polynomial graph.\hfill $\square$
\end{theorem}
\begin{proof}
If $\G$ is quotient polynomial, it is connected and the matrices $\J_0,\ldots,\J_d$ clearly satisfy the conditions  for being an association scheme $\Alg$ with $d$ classes. Conversely, if $\G$ is a connected graph that generates a $d$-class association scheme $\Alg$, its adjacency matrix $\A$ has $d+1$ distinct eigenvalues $\lambda_0,\ldots,\lambda_d$, and $\Alg$ has bases $\I,\A,\ldots,\A_d$ and $\E_0,\ldots, \E_d$ (the minimal idempotents) related by the equalities
$$
\A_i=\sum_{j=0}^d P_{ji} \E_j,\qquad i=0,\ldots,d,
$$
(with $P_{ji}\in \Re$ being the eigenvalues of the scheme, $P_{j1}=\lambda_j$).
Thus, the polynomial $p_i\in \Re_d[x]$ satisfying $p_i(\lambda_j)=P_{ji}$ satisfies $p_i(\A)=\A_i$ for every $i=0,\ldots,d$, and $\G$ is quotient polynomial by Theorem \ref{main-theo}.
\end{proof}

\noindent{\large \bf Acknowledgments.} The author acknowledges the useful comments and suggestions of E. Garriga and J.L.A. Yebra.
This research is supported by the
{\em Ministerio de Ciencia e Innovaci\'on} (Spain) and the {\em European Regional
Development Fund} under project MTM2011-28800-C02-01, and the {\em Catalan Research
Council} under project 2009SGR1387.

\end{document}